\documentclass[12pt]{article}
\usepackage{amsmath,amssymb,amsthm}
\usepackage{graphicx}
\usepackage[utf8x]{inputenc}
  \newtheorem{theorem}{Theorem}

 \newtheorem{lemma}[theorem]{Lemma}
 \newtheorem{proposition}[theorem]{Proposition}
 
 {\theoremstyle{definition}

 }
\newcommand{\N}{\ensuremath{\mathbb N}} 
\newcommand{\R}{\ensuremath{\mathbb R}} 
\newcommand{\Z}{\ensuremath{\mathbb Z}} 
\usepackage{buczo}
\usepackage{tempcrc}

\title{Monotone and convex restrictions of continuous functions}
\author{Zolt\'an Buczolich\thanks{
Research supported by National Research, Development and Innovation Office--NKFIH, Grant 104178.
\newline\indent {\it 2000 Mathematics Subject
Classification:} Primary : 26A15; Secondary : 26A12, 26A51, 28A78.
\newline\indent {\it Keywords:} Hausdorff dimension, Minkowski (box) dimension, typical/generic functions, H\"older spaces.},
Department of Analysis, E\"otv\"os Lor\'and\\
University, P\'azm\'any P\'eter S\'et\'any 1/c, 1117 Budapest, Hungary\\
email: buczo@cs.elte.hu\\
{\tt www.cs.elte.hu/\hbox{$\sim$}buczo}
}
\date{\today}
\begin{document}
\maketitle

\medskip

\begin{center}
{\em Dedicated to Jean-Pierre Kahane on the occasion of his $90^{\text ieth}$ birthday}
\end{center} 


\begin{abstract}
Suppose that $f$ belongs to a suitably defined complete metric space $\cac^{\aaa}$
of H\"older $\aaa$-functions defined on $[0,1]$.
We are interested in whether one can find large (in the sense of Hausdorff, or lower/upper Minkowski dimension) sets  $A\sse [0,1]$
such that $f|_{A}$ is monotone, or convex/concave.
Some of our results are about generic functions in $\cac^{\aaa}$
like the following one: we prove that for a generic $f\in C_{1}^{\aaa}[0,1]$, $0<\aaa<2$ 
  for any $A\sse [0,1]$
 such that  $f|_{A}$ is convex, or concave we have 
$\dimh A\leq \ldimm A\leq \max \{ 0, \aaa-1 \}.$
On the other hand we also have some results about all functions belonging to a certain space. For example the previous result is complemented by the following
one: for $1<\aaa\leq 2$ for any $f\in C^{\aaa}[0,1]$
  there is always a set $A\sse[0,1]$  such that $\dimh A=\aaa-1$ and $f|_{A}$ is convex, or concave on $A$.
\end{abstract}


\section{Introduction}
In \cite{[KK1]} and \cite{[KK2]} J-P. Kahane and Y. Katznelson  showed that
for every real-valued $f\in C[0,1]$ there are closed sets $A\sse [0,1]$
 such that $\dimh A\geq 1/2$ and $f$ is of bounded variation on $A.$
 (We denote by $\dimh A$, $\udimm A$, and $\ldimm A$  the Hausdorff, the upper and lower Minkowski (box) dimension of the set $A$, respectively.)
 This theorem was rediscovered independently in a more general setting by
 A. M\'ath\'e in \cite{[Ma1]}  when only Lebesgue measurability is assumed
 about $f$. 
 Several further variants of restrictions of bounded variation were considered in \cite{[ABMP]}, \cite{[E]} and \cite{[KK1]}, these
include optimality of the bound $\dimh A\geq 1/2$, restrictions of generic continuous functions and functions with values
in $\R^{d}$.

 In \cite{[KK1]} H\"older restrictions were also considered
 for example it was proved that for $0<\aaa<1$ for a continuous $f$  there exists a closed set
 $A$  such that $\dimh A=1-\aaa$ and $f\in C^{\aaa}(A)$. In both cases
 it was illustrated by examples that the dimension bounds on $A$ are best possible.
 In \cite{[KK2]} Kahane and Katznelson showed that for a typical/generic (in the sense of Baire category) $f\in C[0,1]$ for any $0<\aaa<1$ if $f|_{A}$ is in 
 $C^{\aaa}(A)$ then $\ldimm A\leq 1-\aaa$. This result for Hausdorff dimension was obtained by M. Elekes in \cite{[E]}.
 In \cite{[ABMP]} the following generic/typical result was proved
 (see the definition of  $C_{1}^{\aaa}[0,1]$ in Section \ref{*secnota*}).
 Suppose that $0<\bbb<1$. For a generic $f\in {C_{1}^{\bbb}[0,1]}$ 
if $f|_{A}\in C^{\aaa}(A)$ for some $\bbb<\aaa\leq 1$ then $\dimh A\leq 1-\aaa$.
 
Kahane and Katznelson in \cite{[KK1]} also discussed H\"older restrictions
of H\"older functions, for example they showed that there exist functions $f\in C^{\bbb}[0,1]$
 such that if $f|_{A}\in C^{\aaa}(A)$ then $\dimh A\leq \frac{1-\aaa}{1-\bbb}$.
 They also asked whether this result was best possible. It turned out that this
result was not the best possible and a sharp upper bound was provided by  
O. Angel, R. Balka, A. Máthé and Y. Peres in \cite{[ABMP]}.

In \cite{[HL]} P. Humke and M. Laczkovich proved that  if $\fff$ is a porosity premeasure then a typical/generic continuous function on $[0,1]$ intersects every monotone function in a bilaterally strongly $\fff$-porous set. This implies that if the restriction of
a generic continuous function is monotone on a set $A$ then
$\dimh A=0$.
Kahane and Katznelson also considered in \cite{[KK1]} monotone restrictions
of continuous functions and showed that there exists 
$f\in C[0,1]$  such that if $f|_{A}$
is monotone then $\dimh A=0$. It is a natural question
whether such a non-empty compact set $A$ exists at all
for any $f\in C[0,1]$. This question was asked and answered a long time ago.
F. Filipczak proved in \cite{[Fi]} that
for any $f\in C[0,1]$  there exists a perfect,
non-empty set $A\sse [0,1]$  such that $f|_{A}$ is monotone. 
This question of Real Analysis has a graph theoretical interpretation:
Color the complete graph with vertices $x,y\in [0,1]$. An edge
$x\to y$ is red if $\frac{f(y)-f(x)}{y-x}>0$, otherwise it is blue.
{Ramsey's theorem} implies that there exists an infinite $A\sse [0,1]$
 such that all edges between points of $A$ are of the same color.
 In the language of Analysis this means that $f|_{A}$ is monotone.

 We remark that S. Todorčević has some abstract results which offer a criterion for the existence of homogeneously colored perfect sets in colorings of analytic spaces see \cite{[TF]}.

 One can ask similar questions  about convexity, or $n$-convexity.
Recall that the first divided difference is $f[x_{1},x_{2}]=\frac{f(x_{2})-f(x_{1})}
{x_{2}-x_{1}}$,  the second divided difference is $f[x_{1},x_{2},x_3]=\frac{f[x_{2},x_{3}]-f[x_{1},x_{2}]}{x_{3}-x_{1}}$ and if the $(n-1)^{\text{st}}$ divided difference is given then the $n^{\text{th}}$ is 
$$f[x_{1},...,x_n]=\frac{f[x_{2},...,x_{n}]-f[x_{1},...,x_{n-1}]}
{x_{n}-x_{1}}.$$
For $f\in C[0,1]$, color the hypergraph $(x_{1},...,x_{n+1})\in [0,1]^{n+1}$
by red if $f[x_{1},...,x_{n+1}]>0$, by blue otherwise.
Again Ramsey Theory implies that  there exists a  homogenously colored infinite subgraph, that is, there exists an
 infinite $A\sse [0,1]$ such that $f|_{A}$ is $n$-convex, or $n$-concave.
S. Agronski, A. M. Bruckner, M. Laczkovich, and D. Preiss asked in \cite{[ABLP]}
the following question:

{\it Suppose $f\in C[0,1]$ and $n\in \N$. Does there exist $A\sse [0,1]$, non-empty, perfect such that $f|_{A}$ is $n$-convex, or $n$-concave?}

Filipczak's theorem is the $n=1$ case.
Concerning the $n=2$ case, the case of
ordinary convexity, in \cite{[Busconv]} Buczolich proved the following:

{\it For every $f\in C[0,1]$ at least one of the following is true:\\
(i) There exists an interval $I\sse [0,1]$  such that $f|_{I}$ is convex.\\
(ii) There exists an interval $I\sse [0,1]$  such that $f|_{I}$ is concave.\\
(iii) There exist  $A_{1},A_{2}\sse [0,1]$ non-empty perfect such that $f|_{A_{1}}$ is strictly convex and $f|_{A_{2}}$ is strictly concave.}

It is interesting that this result does not hold for higher convexity. 
 A. Olevski\u{\i} in
 \cite{[Ol]} proved that
 there exists a Lipschitz $f:[0,1]\to\R$ 
 such that $f$ is neither $3$-convex, nor $3$-concave on any
non-empty perfect set. 
 
Given the combinatorial/Ramsey theory background P. Erd\H os asked
from the author the following question:

{\it  Suppose that $f:[0,1]\to \R$ is not convex on any
  $r$ element set $3<r\leq \omega$. What can be said about $f$?}

It  was answered in \cite{[Buramsey]}.

The main goal of the current paper is to obtain results about
convex restrictions of functions belonging to different Hölder classes.
Since the derivatives of convex functions are monotone, these
results are going hand in hand with results about monotone
restrictions and in Section \ref{*secmon*} we discuss such results.
In Theorem \ref{*th4} we show that for a generic/typical $f\in \cea$, 
when $0\leq \aaa<1$ if $A\sse [0,1]$ and $f|_{A}$ is monotone then
 $\ldimm A\leq \aaa$.
 It is rather easy to see, and it is at least implicitely well-known, what is stated in Theorem \ref{*th3}:  for any
 $0<\aaa\leq 1$ if $f\in C^{\aaa}[0,1]$ then  there exists 
$A\sse [0,1]$  such that $f|_{A}$ is monotone and $\dimh A \geq \aaa$.
 According to a result from \cite{[ABMP]} (stated as Theorem \ref{*th5a*}
in this paper)
 if $0<\aaa<1$
and  $B(t)$  is a fractional Brownian motion of Hurst index
$\aaa$ then almost surely  $B|_{A}$ is not monotone increasing for any $A$ with $\udimm A>\max\{ 1-\aaa,\aaa \}$.
A fractional Brownian motion of Hurst index
$\aaa$ almost surely belongs to $C^{\aaa-}[0,1]$, but not to
$C^{\aaa}[0,1]$.
In \cite{[ABMP]} for a dense set of $\widehat{\aaa}$s in $[1/2,1]$ examples of self similar functions  $f\in C^{\widehat{\aaa}}[0,1]$ were also provided
for which  $f|_{A}$ is not monotone for any $A$ with $\udimm A>\widehat{\aaa} $.
I learned from R. Balka about an unpublished argument of A. Máthé, which  implies that for any
function $f:[0,1]\to \R$ one can always find a set $A$ such that 
$f|_{A}$ is monotone and $\udimm A\geq 1/2 $.

In Section \ref{*seccvx*} we turn to convex restrictions.
In Theorem \ref{*th11} we see that for a typical $f\in\cea$,
$0\leq \aaa<2$   there is always
 a set $A\sse [0,1]$  such that $f|_A$ is convex and $\udimm A=1$.
 in Theorem \ref{*th6} we "integrate" the result of Theorem \ref{*th3}
 to show that for $1<\aaa\leq 2$ for any  $f\in C^{\aaa}[0,1]$
  there is always a set $A\sse[0,1]$  such that $\dimh A=\aaa-1$ and $f|_{A}$ is convex, or concave on $A$.
In the Theorem \ref{*th12} and Lemma \ref{*lemAB*} we see that the results about generic functions and monotone restrictions can be ``integrated" to obtain results about generic functions and convex restrictions.
In Theorem \ref{*th11b} we prove that for a generic $f\in C_{1}^{\aaa}[0,1]$, $0\leq \aaa<2$ 
  for any $A\sse [0,1]$
 such that  $f|_{A}$ is convex, or concave we have 
$\ldimm A\leq \max \{ 0, \aaa-1 \}.$

We mention in the end of the paper a result (Theorem \ref{*5b*})
according to which there are functions in $f\in C^{\aaa}_{1}[0,1]$,
$1\leq \aaa<2$
  such that 
  $f|_{A}$ is not convex, nor concave for any $A\sse [0,1]$ with $\udimm A>\aaa-1$.
For $3/2\leq \aaa<2$ 
by integrating Fractional Brownian motions of Hurst index
$\aaa-1$ one can obtain functions $f \in C^{\aaa-}[0,1]$
with the property that $f|_{A}$ is not convex, nor concave for any $A\sse [0,1]$ with $\udimm A>\aaa-1$.
By using from \cite{[ABMP]}  the earlier mentioned dense set of $\widehat{\aaa}$s in $[1/2,1]$
taking integrals of the 
corresponding self-similar functions one  can obtain a dense set in $3/2\leq \aaa<2$ and functions $f \in C^{\aaa}[0,1]$
with the property that $f|_{A}$ is not convex, nor concave for any $A\sse [0,1]$ with $\udimm A>\aaa-1$.
For $1\leq \aaa <3/2$ 
it cannot obtained by "integrating"  a theorem about monotone restrictions.
Theorem \ref{*5b*} is proved for $1<\aaa<2$ in \cite{[crm]}.

The authors thanks R. Balka and the referee for several suggestions 
which improved the paper.

\section{Notation and preliminary results}\label{*secnota*}

For $\aaa\geq 0$ if $f:[0,1]\to\R$ is $\lf \aaa \rf$-times differentiable 
(by definition $f^{(0)}=f$)
we put 
\begin{equation}\label{*defcalla*}
\calla(f)=\sup\Big \{ \frac{|f^{(\lf \aaa \rf)}(x)-f^{(\lf \aaa \rf)}(y)|}{|x-y|^{\{ \aaa \}}}:x\not=y,\  x,y\in [0,1] \Big \}.
\end{equation}
\\
By $C^{0}[0,1]$, or $C[0,1]$ we denote the class of continuous
functions on $[0,1]$.

If $0<\aaa<1$ then $f$ is in $C^{\aaa}[0,1]$, if $\calla(f)<\oo$.

For $\aaa=n\in\{ 1,2,... \}$ the function $f$ is in $C^{n}[0,1]$ if
$f^{(n)}$ is continuous on $[0,1]$.
 
If $1<\aaa <2$ then  $f$ is in $ C^{\aaa}[0,1]$ if $f$ is  differentiable, 
and $f^{'}\in C^{\aaa-1}[0,1]$,
that is $\calla(f)<\oo$. 

We denote by $ C^{\aaa-}[0,1]$ 
the set of those functions which are in $C^{\bbb}[0,1]$ for all $\bbb<\aaa$.

The class of Lipschitz functions, the functions for which $${\call}_{0,1}(f)=
\sup\Big \{ \frac{|f(x)-f(y)|}{|x-y|}:x\not=y,\  x,y\in [0,1] \Big \}<+\oo$$ is denoted by $C^{0,1}[0,1]$. While   $C^{1,1}[0,1]$ denotes the class of those $f\in C^{1}[0,1]$,
for which $f'\in  C^{0,1}[0,1]$.

If $f\in C[0,1]= C^{0}[0,1]$ then
$||f||_{0}=\sup_{x\in [0,1]}|f(x)|$. For $\aaa=n\in \N$
we have $||f||_{\aaa}=||f||_{n}=\sup_{j\in\{0,...,n\}}||f^{(j)}||_{0}$.
The open balls in $C^{\aaa}[0,1]$ of radius $r$, centered at $f\in C^{\aaa}[0,1]$,   are denoted by $B_{\aaa}(f,r)$.

Some special subspaces. For $0\leq \aaa$, $\aaa\not\in \N$  $f$ is in $C_{1}^{\aaa}[0,1]$ if
$\calla(f)\leq 1$, that is $|f^{(\lf \aaa  \rf)}(x)-f^{(\lf \aaa  \rf)}(y)|\leq |x-y|^{\aaa}$ for all $x,y\in [0,1]$.
It is clear that $C_{1}^{\aaa}[0,1]$ is a closed, separable and  complete subspace of
$C^{\lf \aaa \rf}[0,1]$ when we use the subspace metric $||f-g||_{\lf \aaa \rf}$
for $f,g\in C_{1}^{\aaa}[0,1]$. 
When working in these spaces we will keep the notation 
$B_{\aaa}(f,r)$ for the open balls in these subspaces.

Suppose $A\sse \R^{n}$.
The system $\{ U_{j} \} \text{ is a } {\ddd\text{-cover of }A}$
 if $\diam(U_{j})<\ddd$  for all $ j$, and $A\sse\cup_{j}U_{j}$.

The $\aaa$-dimensional Hausdorff measure (see its definition for example in \cite{Fa1}) is denoted by $\cah^{\aaa}$. Recall that the Hausdorff dimension of $A\sse \R^{n}$
is given by
 $$\ds \dimh A=\inf\{ \aaa : \cah^{\aaa}(A)=0\}=$$ $$\inf\{\aaa:\ex C_{\aaa}>0,\ \ax\ddd>0,\  \ex \{ U_{j} \} \text{ a } \ddd\text{-cover of }A  \text{ s.t.} \sum_{j}(\diam(U_{j}))^{\aaa}<C_{\aaa} \}.$$

Given an integer $k\geq 2$ and a set $A\sse [0,1]$ we put
$$\can_{k,\ell }(A)=\#\Big \{ j\in \Z:A\cap \Big [ \frac{j-1}{k^{\ell }},\frac{j}{k^{\ell }} \Big ]\not=\ess \Big \}.$$
Most often we use the $k=2$, or $k=10$ cases, $\cantl(A)$, or $\cankl(A)$.

The upper and lower Minkowski (or box) dimension of $A$ is defined as
\begin{equation}\label{*minkdi}
\udimm A=\limsup_{\ell \to\oo}\frac{\log \can_{k,\ell }(A)}{\ell \log k}\text{ and }
\ldimm A=\liminf_{\ell \to\oo}\frac{\log \can_{k,\ell }(A)}{\ell \log k}.
\end{equation}
It is well-known that for any $k$ we obtain the same value and $\dimh A\leq\ldimm A\leq \udimm A$.


Recall Proposition 2.2 of \cite{Fa1}:
\begin{proposition}\label{*profalconer}
Let $F\sse \R^{m}$ and $f:F\to \R^{m}$ be a mapping  such that 
$$|f(x)-f(y)|\leq c|x-y|^{\aaa}\qquad (x,y\in F)$$
for constants $c>0$ and $\aaa>0$.
Then for each $s$
$$\cah^{s/\aaa}f(F)\leq c^{s/\aaa}\cah^{s}(F).$$
\end{proposition}

\section{Results about monotone restrictions}\label{*secmon*}
 
 \begin{theorem}\label{*th4}
 Suppose $0\leq \aaa<1.$
 There exists a dense $G_{\ddd}$ set $\cag \sse \cea$, ($\cag \sse C[0,1]$ when $\aaa=0$) such that 
 if $f\in \cag$, $A\sse [0,1]$ and $f|_{A}$ is monotone then
 $\dimh A\leq \ldimm A\leq \aaa$.
 \end{theorem}

 \begin{proof}
 We prove the existence of a dense $G_{\ddd}$  set $\cag \sse \cea$,  such that 
 if $f\in \cag$, $A\sse [0,1]$ and $f|_{A}$ is monotone increasing then
 $\ldimm A\leq \aaa$. A similar theorem is valid for monotone decreasing functions
 and the intersection of two dense $G_{\ddd}$ sets is still dense $G_{\ddd}$
 and this yields the theorem.
 
One can easily see that 
 $C^{\oo}$ functions $f$ satisfying
 $\calla(f)<1$  are dense in $\cea$ when $0<\aaa<1$.
 
 One can select a set $\{ {{f}}_{n} \}_{n=1}^{\oo}$ of $C^{\oo}$ functions which is dense in $C_{1}^{\aaa}[0,1]$, (in $C[0,1]$ when $\aaa=0$). 
 We choose 
 \begin{equation}\label{*G2*e}
 \Xi_{n}> n, \ \Xi_{n}\in \N \text{  such that  }
 \end{equation}
 \begin{equation}\label{*G2*a}
 |{{f}}_{n}' (x)|<\Xi_{n} \text{ holds for any }x\in[0,1].
 \end{equation}
 We also suppose that 
 \begin{equation}\label{*G2b*a}
 \calla({{f}}_{n})<1\text{ and }\frac{1}{\Xi_{n}}<1-\calla({{f}}_{n}),
 \end{equation}
 (for $\aaa=0$ we do not need this assumption).
 We put
 \begin{equation}\label{*G7*a}
 M_{m,n}=2{\Xi_{n}(\Xi_{n}+10)(m+n)}\text{ and }M_{m,n}'=\Big \lf M_{m,n}\aaa+\frac{\log M_{m,n}}{\log 2} \Big \rf.
 \end{equation}
 We choose ${M}_{\aaa}$
 such that  if $m\geq {M}_{\aaa}$ then by also recalling \eqref{*G2*e} we have
\begin{equation}\label{*G2*mmn}
M_{m,n}'<M_{m,n}-10 \text{ for any }n.
\end{equation} 
Next we suppose that $m\geq {M}_{\aaa}$.

 Next we define a function $\fff_{m,n}$ periodic by $2^{-M_{m,n}'}$. It is sufficient to define
 this function on $[0,2^{-M_{m,n}'}]$. Set
 \begin{equation}\label{*G2*c}
 \fff_{m,n}(0)=\fff_{m,n}(2^{-M_{m,n}'})=0\text{ and}
 \end{equation}
 \begin{equation}\label{*G8*b}
 \fff_{m,n}(2^{-M_{m,n}'}-2^{-M_{m,n}})=\frac{-1}{\Xi_{n}}2^{-M_{m,n}\aaa}
 \end{equation}
 and we suppose that $\fff_{m,n}$ is linear on the intervals
 $[0,2^{-M_{m,n}'}-2^{-M_{m,n}}]$ and $[2^{-M_{m,n}'}-2^{-M_{m,n}},2^{-M_{m,n}'}]$.

This way for any $x\in (0,2^{-M_{m,n}'}-2^{-M_{m,n}})+j\cdot 2^{-M_{m,n}'},$ $j\in \Z$
\begin{equation}\label{*G8*c}
\fff_{m,n}'(x)=\frac{-2^{-M_{m,n} \aaa}}{\Xi_{n}(2^{-M_{m,n}'}-2^{-M_{m,n}})}<\frac{-2^{-M_{m,n}\aaa}\cdot 2^{M_{m,n}'}}{\Xi_{n}}<
\end{equation}
$$\frac{-2^{-M_{m,n}\aaa +M_{m,n}\aaa+\frac{\log M_{m,n}}{\log 2}-1}}{\Xi_{n}}=-(\Xi_{n} +10)(m+n)<-(\Xi_{n} +10) .$$
Otherwise we always have 
  \begin{equation}\label{*95*b}
  |\fff_{m,n}'(x)|\leq \frac{1}{\Xi_{n}}2^{-M_{m,n}\aaa}\cdot 2^{M_{m,n}}
  \text{ wherever it exists.}
  \end{equation}
 
 We put 
 $
 f_{m,n}(x)={{f}}_{n}(x)+\fff_{m,n}(x).
 $
 
 Next we show that the functions $f_{m,n}$ are in $C_{1}^{\aaa}[0,1]$ when $0<\aaa<1$.
 Using \eqref{*G2*c}, \eqref{*G8*b} and \eqref{*95*b} for any $x,y\in [0,1]$
 \begin{equation}\label{*G9*a}
 \frac{|\fff_{m,n}(x)-\fff_{m,n}(y)|}{|x-y|^{\aaa}}\leq
 \max_{0<|x'-y'|\leq 2^{-M_{m,n}}}\frac{\frac{1}{\Xi_{n}}2^{-M_{m,n}\aaa}\cdot 2^{M_{m,n}}|x'-y'|}{|x'-y'|^{\aaa}}\leq
 \end{equation}
 $$\frac{1}{\Xi_{n}}2^{-M_{m,n}\aaa+M_{m,n}-(1-\aaa)M_{m,n}}=\frac{1}{\Xi_{n}}<1-\calla({{f}}_{n}).$$
 This implies that
 \begin{equation}\label{*G9*b}
 \frac{|f_{m,n}(x)-f_{m,n}(y)|}{|x-y|^{\aaa}}\leq \frac{|{{f}}_{n}(x)-{{f}}_{n}(y)|}{|x-y|^{\aaa}}+
 (1-\calla({{f}}_{n}))\leq 1.
 \end{equation}
 When $\aaa=0$ then it is clear that $f_{m,n}\in C[0,1]$.
 
 By \eqref{*G2*e} and \eqref{*G8*b}
 \begin{equation}\label{*G10*a}
 |f_{m,n}(x)-{{f}}_{n}(x)|=|\fff_{m,n}(x)|\leq \frac{1}{\Xi_{n}}2^{-M_{m,n}\aaa}< \frac{1}{n},\quad \ax x\in[0,1].
 \end{equation}
 Hence, the density of ${{f}}_{n}$ in $C_{1}^{\aaa}[0,1]$ 
 (in $C[0,1]$ when $\aaa=0$)
 implies that of $f_{m,n}$.
 Put
 \begin{equation}\label{*G4*d}
 \ddd_{m,n}=2^{-M_{m,n}-1},\ {\cal G}_{m}=\cup_{n=1}^{\oo}B_{\aaa}(f_{m,n},\ddd_{m,n})
 \text{ and }\cag=\cap_{m={M}_{\aaa}}^{\oo}{\cal G}_{m}.
 \end{equation}
 
 The $\cag_{m}$'s are dense and open in $C_{1}^{\aaa}[0,1]$, (in $C[0,1]$ when $\aaa=0$) and $\cag$ is dense $G_{\ddd}$.

 Suppose $f\in \cag$ and $A\sse[0,1]$ is  such that $f|_{A}$ is monotone increasing.

Then for each $m$ there exists $n(m)$  such that 
$
f\in B_{\aaa}(f_{m,n(m)},\ddd_{m,n(m)}).
$

Suppose that for a $j$   there exist $p_{1}<p_{2}$,
$p_{1},p_{2}\in [0,2^{-M_{m,n(m)}'}-2^{-M_{m,n(m)}}]+j\cdot 2^{-M_{m,n(m)}'}$
with
\begin{equation}\label{*G4*a}
p_{2}-p_{1}\geq 2^{-M_{m,n(m)}}\text{ and }f(p_{2})\geq f(p_{1}).
\end{equation}
Then by \eqref{*G4*d}
\begin{equation}\label{*G5**a}
f_{m,n(m)} (p_{2})\geq f_{m,n(m)}(p_{1})-2\ddd_{m,n(m)} =
f_{m,n(m)}(p_{1})-2^{-M_{m,n(m)}}.
\end{equation}
 On the other hand, by \eqref{*G2*a}, \eqref{*G8*c} and \eqref{*G4*a}
 $$f_{m,n(m)}(p_{2})\leq f_{m,n(m)}(p_{1})-10\cdot 2^{-M_{m,n(m)}},$$
this contradicts \eqref{*G5**a}. Hence $A$ can intersect at most two intervals
of the form $[(j'-1)\cdot 2^{-M_{m,n(m)}},j'\cdot 2^{-M_{m,n(m)}}]$ 
 in an interval $[0,2^{-M_{m,n(m)}'}-2^{-M_{m,n(m)}}]+j\cdot 2^{-M_{m,n(m)}'}.$
 
 Therefore, $A\cap [0,2^{-M_{m,n(m)}'}]+j\cdot 2^{-M_{m,n(m)}'}$ for any $j\in\Z$ can be covered by no more than three
 intervals of the form $[(j'-1)\cdot 2^{-M_{m,n(m)}},j'\cdot 2^{-M_{m,n(m)}}]$.
 Thus, by  \eqref{*G7*a}
 $$\can_{2,M_{m,n(m)}}(A)\leq 3\cdot 2^{M_{m,n(m)}'}\leq 3\cdot 
 2^{M_{m,n(m)}\aaa+\frac{\log(M_{m,n(m)})}{\log 2}}.$$
 Hence,
 $$\liminf_{m\to\oo} \frac{\log \can_{2,M_{m,n(m)}}(A)}{M_{m,n(m)} \log 2}\leq\aaa,$$
 and this implies $\ldimm A\leq \aaa.$
\end{proof}

\begin{theorem}\label{*th3}
Suppose $0<\aaa\leq 1$. If $f\in C^{\aaa}[0,1]$ then  there exists 
$A\sse [0,1]$  such that $f|_{A}$ is monotone and $\dimh A \geq \aaa$.
\end{theorem}

\begin{proof}
Suppose $0<\aaa<1$. If $f$ is constant then we are done.
If   there exist $a<b$  such that $f(a)<f(b)$ then we can consider the function $g:[a,b]\to\R$  such that $$g(x)=\max\{ f(t):t\in [a,x] \}.$$
Then $g(x)$ is monotone increasing and $C^{\aaa}$. Let $A=\{ x:g(x)=f(x) \}.$ The set $A$ is closed and $g$ is constant on the intervals contiguous to $A$ and $g(A)=[f(a),f(b)]$. If we had $\dimh A<\aaa$ then 
by Proposition \ref{*profalconer} we would obtain $\dimh g(A)<1$,
a contradiction.
If $\aaa=1$ then $f'\in C[0,1]$ and is either identically zero or there is an interval where it is of constant sign.
\end{proof}

 The next theorem is from \cite[Corollary 1.4 and Proposition 1.5]{[ABMP]}
 \begin{theorem}\label{*th5a*}
 Let $0<\aaa<1$
and let $\ds \{ B(t):t\in[0,1] \}$  be a fractional Brownian motion of Hurst index
$\aaa$. Then almost surely for any $A$ with $\udimm A>\max\{ 1-\aaa,\aaa \}$, $B|_{A}$ is not monotone increasing.
If $\car=\{ t\in[0,1] : B(t)=\max_{s\in[0,t]}B(s) \}$
the set of record times of $B$ then almost surely $\dimh\car=\udimm \car=\aaa$.
 \end{theorem}

 The functions $B(t)$ in the above theorem belong to $C^{\aaa-}[0,1]\sm C^{\aaa}[0,1]$.
Since $\max\{ 1-\aaa,\aaa \}\geq 1/2$ one might wonder whether for 
$\aaa<1/2$ one can obtain better estimates on the upper Minkowski dimesion of sets where $C^{\aaa}$ functions are monotone. This is not the case, since an unpublished result of A. M\'ath\'e shows that for any
function $f:[0,1]\to \R$ one can always find a set $A$ such that 
$f|_{A}$ is monotone and $\udimm A\geq 1/2 $.

\section{Results about convex restrictions}\label{*seccvx*}

 
 \begin{lemma}\label{*lemB4*}
 Suppose that $a<b$, $\xi_{0}''\geq\xi_{0}>0$,  $f_{0}:[a,b]\to\R$,
 \begin{equation}\label{*B4*e}
 f_{0}''(x)=\xi_{0}''\geq \xi_{0}>0 \text{ for all }x\in [a,b]
 \end{equation}
 and $f\in B_{0}(f_{0},\ddd_{0})$. Let
 \begin{equation}\label{*B4*b}
 E_{f}\text{ denote the convex hull of }\{ (x,y): x\in [a,b], \ y\geq f(x) \}
 \end{equation}
and put
\begin{equation}\label{*B4*c}
g(x)=\min\{ y:(x,y)\in E_{f} \}\text{ for }x\in [a,b].
\end{equation}
Set
\begin{equation}\label{*B4*d}
A=\{ x\in [a,b]:f(x)=g(x) \}.
\end{equation}
Suppose $x_{1},x_{2}\in A$ satisfy $x_{1}<x_{2}$ and $(x_{1},x_{2})\cap A=\ess$.
Then
\begin{equation}\label{*B4*a}
x_{2}-x_{1}<4\sqrt{\frac{\ddd_{0}}{\xi_{0}}}.
\end{equation}
 \end{lemma}
 
 \begin{proof}
 Since $(x_{1},x_{2})\cap A=\ess$ we have
 \begin{equation}\label{*B5*a}
 g(x)=g(x_{1})+\frac{g(x_{2})-g(x_{1})}{x_{2}-x_{1}}(x-x_{1})\text{ for any }
 x\in(x_{1},x_{2}).
 \end{equation}
 
 Denote by $E_{f_{0}-\ddd_{0}}$ the convex hull of
 $\ds \{ (x,y):x\in [a,b],\ y\geq f_{0}(x)-\ddd_{0} \}$.

 Since $f_{0}-\ddd_{0}$ is convex $f_{0}(x)-\ddd_{0}=\min\{ y:(x,y)\in E_{f_{0}-\ddd_{0}} \}.$ It is also clear that $E_{f}\sse E_{f_{0}-\ddd_{0}}$
 and
 $g(x)>f_{0}(x)-\ddd_{0}$ holds on $[a,b]$.
 Since $f_{0}(x)+\ddd_{0}>f(x)\geq g(x)$ we also
 have $g(x)<f_{0}(x)+\ddd_{0}$.
 
 Put $x_{3}=\frac{x_{1}+x_{2}}{2}.$ Using \eqref{*B4*e} 
 $$f_{0}'(x_{3})=\frac{f_{0}(x_{2})-f_{0}(x_{1})}{x_{2}-x_{1}}
 \text{ and }f_{0}(x_{1})=f_{0}(x_{3})+f_{0}'(x_{3})(x_{1}-x_{3})+
 \frac{\xi_{0}''}{2}(x_{1}-x_{3})^{2}.$$
 Substituting the value of $f_{0}'(x_{3})$ in the last equation
 and rearranging we obtain
 \begin{equation}\label{*B5*b}
 f_{0}(x_{1})+\frac{f_{0}(x_{2})-f_{0}(x_{1})}{x_{2}-x_{1}}(x_{3}-x_{1})-f_{0}(x_{3})=
 \frac{f_{0}(x_{1})+f_{0}(x_{2})}{2}-f_{0}(x_{3})=
 \end{equation}
 $$\frac{\xi_{0}''}{2}(x_{3}-x_{1})^{2}.
 $$
 By \eqref{*B5*a} we have
 \begin{equation}\label{*B5*c}
 g(x_{3})=\frac{g(x_{1})+g(x_{2})}{2}.
 \end{equation}
 Using that $|f_{0}-g|<\ddd_{0}$ we deduce from \eqref{*B5*b}
 and \eqref{*B5*a} that
 $$\ddd_{0}>|f_{0}(x_{3})-g(x_{3})|=\Big |\frac{f_{0}(x_{1})+f_{0}(x_{2})}{2}
-\frac{\xi_{0}''}{2}(x_{3}-x_{1})^{2}-\frac{g(x_{1})+g(x_{2})}{2} \Big |\geq$$
$$ \frac{\xi_{0}''}{2}(x_{3}-x_{1})^{2}-\ddd_{0}.$$
 Hence,
 $$\frac{x_{2}-x_{1}}{2}=x_{3}-x_{1}<\sqrt{\frac{4\ddd_{0}}{\xi_{0}''}}$$
 which implies \eqref{*B4*a}.
 \end{proof}

 \begin{theorem}\label{*th11}
 Suppose $0\leq \aaa<2$. There exists a dense $G_{\ddd}$ set
 $\cag$ in $\cea$  such that for any $f\in \cag$ there exists
 a closed set $A\sse [0,1]$  on which $f|_A$ is convex and $\udimm A=1$.
 (When $\aaa=0$, or $1$ then we can use $C^{\aaa}[0,1]$
 instead of $\cea$.)
 \end{theorem}
 
 \begin{proof}
The strategy of our proof is the following. 
We obtain dense open sets of functions ${\cal G}_{m}$ defined
in \eqref{*B2*b}. In the definition of these open sets
balls centered at functions $f_{n}$ and of radius $\ddd_{m,n}$ 
are used.
These functions $f_{n}$ will have piecewise linear, locally non-constant
first derivatives and there will be a point 
${\widetilde{x}}_{n}\in (0,1)$
 such that  $f_{n}''({\widetilde{x}}_{n})$ is positive. 
 This means that on suitably small subintervals, first on $[a_{1},b_{1}]$,
 later on $[a_{4,m},b_{4,m}]$
Lemma \ref{*lemB4*} will be applicable and the choice of  
$\ddd_{m,n}$ can be made in \eqref{*B3*b} in  a way that by
Lemma \ref{*lemB4*} if we take $f\in \cap_{m} B_{\aaa}(f_{n(m)},\ddd_{m,n(m)})$ and define auxiliary functions by taking 
boundaries of convex hulls, as in  \eqref{*B4*b} an \eqref{*B4*c}
we can obtain sets $A$ with very small gaps (see \eqref{*B7*a} and \eqref{*B14*c}) on which $f$ is convex and this will yield our estimate
about the upper Minkowski dimension.
There is an additional technical difficulty arising from the fact
that we need to make sure that the small intervals $[a_{4,m},b_{4,m}]$
are selected in a way that the boundary of the convex hull defined by the restriction
of $f$ onto $[a_{4,m},b_{4,m}]$ will be on the boundary  of the convex hull defined by the restriction
of $f$ onto $[a_{1},b_{1}]$ see the remark after \eqref{*B14*b}.

Next we turn to details of the proof.
 Suppose $0\leq \aaa<2$ is fixed. (When $\aaa=0$, or $1$ then we use
 $C^{\aaa}[0,1]$
 instead of $\cea$ in the next proof and we do not need to make any assumptions about $\calla$.)
 
 Given any $\eee>0$ and $f\in \cea$ one can find ${\overline{f}}\in
 C^{1,1}[0,1]$  such that  ${\overline{f}}'$ is piecewise linear,  $\calla({\overline{f}})<1$ and ${\overline{f}}\in B_{\aaa}(f,\eee)$.
 Select $x_{0}\in (0,1)$ and $\ddd_{0}>0$
  such that ${\overline{f}}'$ is linear on $[x_{0}-\ddd_{0},x_{0}+\ddd_{0}]\sse (0,1)$ and hence ${\overline{f}}''(x_{0})$ exists.
  
  Denote by $\fff_{c}$ the function which equals $c$ if $x\not\in [x_{0}-c,x_{0}+c]$ and $2|{\overline{f}}''(x_{0})|$ if $x\in [x_{0}-c, x_{0}+c]$.
  Set $\FFF_{c}(x)=\int_{0}^{x}\int_{0}^{t}\fff_{c}(u)du \, dt$.
  Then $\FFF_{c}'$ is piecewise linear.
  Moreover, 
  \begin{equation}\label{*fffc*}
  \lim_{c\to0+}||\FFF_{c}||_{0}= \lim_{c\to0+}||\FFF_{c}'||_{0}= \lim_{c\to0+}\calla(\FFF_{c})=0.
  \end{equation}

 Hence one can choose a sufficiently small $c>0$  such that 
 $\widetilde{f}={\overline{f}}+\FFF_{c}$ has piecewise linear derivative, 
 there is no interval on which $\widetilde{f}'$
is constant, $\widetilde{f}''(x_{0})>0$ and $\widetilde{f}\in B_{\aaa}(f,\eee).$

By the above remarks one can  select a set $\{ f_{n}:n=1,... \}$
 which is dense in $\cea$ and consists of $C^{1}[0,1]$
 functions  such that  $f_{n}'$ is piecewise linear, there is no interval on which
 $f_{n}'$ is constant and for any $n$
 \begin{equation}\label{*B2*a}
 \text{  there exists  }{\widetilde{x}}_{n}\in (0,1)
\text{ such that } f_{n}''({\widetilde{x}}_{n})>0. 
 \end{equation}
For a given $m\in\N$ first we define dense open sets ${\cal G}_{m}$
in $C^{\aaa}_{1}[0,1]$.

Suppose that $m$ is fixed. We want to select $\ddd_{m,n}>0$ and define
\begin{equation}\label{*B2*b}
{\cal G}_{m}=\bigcup_{n}B_{\aaa}(f_{n},\ddd_{m,n})\text{ and }\cag=\bigcap_{m=1}^{\oo}{\cal G}_{m}.
\end{equation}
 Then, clearly the sets ${\cal G}_{m}$ are dense and open, and $\cag$ is dense $G_{\ddd}$.

 Since $f_{n}'$ is piecewise linear and there is no interval on which it is constant we can select
 $0<\xi_{n}<1<\Xi_{n}$  such that 
 \begin{equation}\label{*B3*a}
 \xi_{n}<|f_{n}''(x)|<\Xi_{n}
 \end{equation}
 at any $x\in [0,1]$ where $f_n'$ is locally linear.
 We can also suppose that any maximal interval on which
 $f_{n}'$ is linear is of length at least  $d_{n}<1/n$.
 
 We need to select $\ddd_{m,n}>0$ for any $m,n\in \N$. First we suppose that
 \begin{equation}\label{*B3*b}
 \ddd_{m,n}\leq \frac{\xi_{n}\cdot d_{n}^{2}}{m\cdot 100^{2}}.
 \end{equation}
 Later we also need to assume \eqref{*B10*b} and \eqref{*63b*}.
 
 Assume $f\in\cag$. Then there exists a sequence $n(m)$
 such that $f\in B_{\aaa}(f_{n(m)},\ddd_{m,n(m)}).$
By property \eqref{*B2*a}  there exists an interval $[a_{1},b_{1}]\sse [0,1]$
 such that $b_{1}-a_{1}\geq d_{n(1)}$ and
\begin{equation}\label{*B6*a}
\xi_{n(1)}''\defeq f_{n(1)}''(x)>\xi_{n(1)}\text{ for any }x\in (a_{1},b_{1}).
\end{equation} 
We define $E_{f} $, $g$ and $A$ as in 
\eqref{*B4*b}, \eqref{*B4*c} and \eqref{*B4*d} in Lemma
\ref{*lemB4*} using $[a_{1},b_{1}]$ instead of $[a,b]$. 
Then $g$ is convex on $[a_{1},b_{1}]$ and $f|_{A}=g|_{A}$
is also convex. We need to show that $\udimm A=1$.

Since the points $(a_{1},f(a_{1}))$ and $(b_{1},f(b_{1}))$
are extremal points of $E_{f} $ we have $a_{1},b_{1}\in A$.
Next we apply Lemma \ref{*lemB4*} with
$a=a_{1}$, $b=b_{1}$, $\xi_{0}=\xi_{n(1)}$, $\xi_{0}''=\xi_{n(1)}''$,
$f_{0}=f_{n(1)}$, $\ddd_{0}=\ddd_{1,n(1)}$ and $f=f$.

We infer that if $a_{1}\leq x_{1}<x_{2}\leq b_{1}$ and $(x_{1},x_{2})\cap A=\ess$
then by \eqref{*B4*a} and \eqref{*B3*b}
\begin{equation}\label{*B7*a}
x_{2}-x_{1}< 4\sqrt{\frac{\ddd_{1,n(1)}}{\xi_{n(1)}}}\leq 4\frac{d_{n(1)}}{100}\leq
\frac{b_{1}-a_{1}}{25}.
\end{equation}

We denote by $E_{f,m}$ the convex hull of $\ds \{ (x,y):x\in [a_{1},b_{1}],\  y\geq f_{n(m)}(x) \}$ and $g_{m}(x)=\min\{ y:(x,y)\in  E_{f,m} \}$ for $x\in[a_{1},b_{1}]$
and $A_{m}=\{ x\in [a_{1},b_{1}]:f_{n(m)}(x)=g_{m}(x) \}.$ 

Since $f\in B_{\aaa}(f_{n(m)},\ddd_{m,n(m)})$ we have
\begin{equation}\label{*B7*b}
E_{f,m}+(0,\ddd_{m,n(m)})\sse E_{f}  \sse E_{f,m}-(0,\ddd_{m,n(m)})
\end{equation}
and hence
\begin{equation}\label{*B7*c}
g_{m}(x)-\ddd_{m,n(m)}\leq g(x)\leq g_{m}(x)+\ddd_{m,n(m)}\text{ for }x\in [a_{1},b_{1}].
\end{equation}
Since $f_{n(m)}'$ is piecewise linear, one can easily see that $g_{m}\in C^{1}[a_{1},b_{1}]$ and at points $x\in A_{m}$ we have
$f_{n(m)}'(x)=g_{m}'(x) $ while $g_{m}$ is locally linear at any $x\in (a_{1},b_{1})
\sm A_{m}$. 
Indeed, $g_{m}$ is convex and at $x\in A_{m}$ the one-sided derivatives 
 $g'_{m,-}(x)$ and $g'_{m,+}(x)$ should both coincide
with $f_{n(m)}'(x)$. Hence at accumulation points of
$A_{m}$ the restriction of $g_{m}'$ onto $A_{m}$ is continuous.
If $(\aaa_{m},\bbb_{m})$ is an interval contiguous to $A_{m}$
then $g_{m}'(x)=g_{m}'(\aaa_{m})=f_{n(m)}'(\aaa_{m})=g_{m}'(\bbb_{m})=
f_{n(m)}'(\bbb_{m})$ should hold for any $x\in (\aaa_{m},\bbb_{m})$.
This easily implies that $g_{m}\in C^{1}[a_{1},b_{1}].$
The above argument also shows that if $g_m'$ is non-constant on an interval
then $A_{m}$ should have points in this interval.

One can deduce from \eqref{*B3*a} that
\begin{equation}\label{*B8*a}
0\leq g_{m}'(y)-g_{m}'(x)<\Xi_{n(m)}\cdot (y-x)\text{ for any }x,y\in[a_{1},b_{1}],\ 
x<y.
\end{equation}
For any $m$ and $n$ we select
a sufficiently small $\DDD_{m,n}>0$
satisfying \eqref{*Dmnr1*}, \eqref{*B12*b} and \eqref{*B13*a},
after the $\DDD_{m,n}$'s are fixed we can select $\ddd_{m,n}$.
Suppose $x,x\pm\DDD_{m,n(m)}\in[a_{1},b_{1}]$.
By \eqref{*B7*c}
\begin{equation}\label{*B9*a}
\Big | \frac{g(x+\DDD_{m,n(m)})-g(x)}{\DDD_{m,n(m)}} - \frac{g_{m}(x+\DDD_{m,n(m)})-g_{m}(x)}{\DDD_{m,n(m)}}\Big |<\frac{2\ddd_{m,n(m)}}{\DDD_{m,n(m)}}.
\end{equation}
Since $\ds \frac{g(x+\DDD_{m,n(m)})-g(x)}{\DDD_{m,n(m)}}\geq g'_{+}(x)$ and
$$\frac{g_{m}(x+\DDD_{m,n(m)})-g_{m}(x)}{\DDD_{m,n(m)}}\leq g_{m}'(x+\DDD_{m,n(m)})<
g_{m}'(x)+\DDD_{m,n(m)}\Xi_{n(m)}$$
we have by \eqref{*B9*a}
\begin{equation}\label{*B9*b}
g_{+}'(x)<g_{m}'(x)+\DDD_{m,n(m)} \Xi_{n(m)} + \frac{2\ddd_{m,n(m)}}{\DDD_{m,n(m)}},
\end{equation}
and similarly
\begin{equation}\label{*B10*a}
g_{+}'(x)\geq g_{-}'(x)>g_{m}'(x)-\DDD_{m,n(m)} \Xi_{n(m)} - \frac{2\ddd_{m,n(m)}}{\DDD_{m,n(m)}}.
\end{equation}

If we assume that
\begin{equation}\label{*B10*b}
\ddd_{m,n}\leq \DDD_{m,n}^{2}
\end{equation}
then \eqref{*B9*b} and \eqref{*B10*a} imply that
\begin{equation}\label{*B10*c}
|g'_{\pm}(x)-g'_{m}(x)|<2\DDD_{m,n(m)}(\Xi_{n(m)}+1)
\end{equation}
$\text{ for any }
x\in[ a_{1}+\DDD_{m,n(m)},b_{1}-\DDD_{m,n(m)}].$

Set $a_{2}=a_{1}+\frac{b_{1}-a_{1}}{25}$ and $b_{2}=b_{1}-\frac{b_{1}-a_{1}}{25}.$
By \eqref{*B7*a}, $A\cap (a_{2},b_{2})\not=\ess$.
We can suppose that
\begin{equation}\label{*Dmnr1*}
\DDD_{1,n(1)}<\frac{d_{n(1)}}{25}\leq \frac{b_{1}-a_{1}}{25}
\end{equation}
we use \eqref{*B10*c} to obtain that
\begin{equation}\label{*B11*a}
|g_{\pm}'(a_{2})-g_{1}'(a_{2})|<2\DDD_{1,n(1)}\cdot (\Xi_{n(1)}+1) 
\text{ and }
\end{equation}
$$|g_{\pm}'(b_{2})-g_{1}'(b_{2})|<2\DDD_{1,n(1)}\cdot (\Xi_{n(1)}+1).$$
Since by \eqref{*B6*a}, $f_{n(1)}''(x)=\xi_{n(1)}''>0$ for $x\in(a_{1},b_{1})$
the function $f_{n(1)}$ is convex on $(a_{1},b_{1})$,  hence  
$f_{n(1)}(x)=g_{1}(x)$ for all $x\in [a_{1},b_{1}]$ and
\begin{equation}\label{*B11*b}
g_{1}'(b_{2})=g_{1}'(a_{2})+\xi_{n(1)}''(b_{2}-a_{2})\geq 
g_{1}'(a_{2})+\xi_{n(1)}'' \frac{d_{n(1)}}{2}.
\end{equation}
Therefore, by \eqref{*B11*a}
\begin{equation}\label{*B12*a}
g'_{-}(b_{2})-g'_{+}(a_{2})>\xi_{n(1)}'' \frac{d_{n(1)}}{2}-4\DDD_{1,n(1)}(\Xi_{n(1)} +1).
\end{equation}
We can suppose that for any $m$ and $n$
\begin{equation}\label{*B12*b}
\DDD_{m,n} < \frac{\xi_{n}'' d_{n}}{16(\Xi_{n}+1)}.
\end{equation}
Thus \eqref{*B11*b} and \eqref{*B12*a} imply that
\begin{equation}\label{*B12*d}
g_{-}'(b_{2})-g_{+}'(a_{2})> \frac{\xi_{n(1)}'' d_{n(1)}}{4}.
\end{equation}
Suppose 
\begin{equation}\label{*B12*c}
m>\frac{8}{\xi_{n(1)}'' d_{n(1)}}.
\end{equation}
Then $[a_{2},b_{2}]$ can be divided into no more than
$\lc \frac{1}{d_{n(m)}} \rc<\frac{2}{d_{n(m)}}$ many intervals
on which $f'_{n(m)}$ is linear, hence there is an
$[a_{3,m},b_{3,m}]\sse [a_{2},b_{2}]$ on which $f_{n(m)}'$ is linear
and using \eqref{*B12*d} and \eqref{*B12*c}
\begin{equation}\label{*B13*d}
g_{-}'(b_{3,m})-g'_{+}(a_{3,m})>\frac{\xi''_{n(1)}d_{n(1)}}{4}\cdot \frac{d_{n(m)}}{2}>\frac{d_{n(m)}}{m}.
\end{equation}
Put
\begin{equation}\label{*B13*b}
a_{4,m}=\min \{ A\cap [a_{3,m},b_{3,m}] \}\text{ and }
b_{4,m}=\max \{ A\cap [a_{3,m},b_{3,m}] \}.
\end{equation}
Since $g$ is linear on $[a_{3,m},a_{4,m}]$
and on $[b_{4,m},b_{3,m}]$ when $a_{3,m}\not=a_{4,m}$, or $b_{3,m}\not=b_{4,m}$ we have
\begin{equation}\label{*B13*c}
g_{-}'(a_{4,m})=g_{+}'(a_{3,m})\text{ and }
g_{+}'(b_{4,m})=g_{-}'(b_{3,m}).
\end{equation}
We can also suppose that 
\begin{equation}\label{*B13*a}
\DDD_{m,n}<\frac{d_{n}}{8m(\Xi_{n}+1)}.
\end{equation}

By \eqref{*B10*c}, \eqref{*B13*d}, \eqref{*B13*c} and \eqref{*B13*a}
\begin{equation}\label{*B14*a}
g'_{m}(b_{4,m})-g'_{m}(a_{4,m})>\frac{d_{n(m)}}{2m}.
\end{equation}
This implies that $g_{m}'$ is non-constant on $[a_{4,m},b_{4,m}]$.
Hence, as we observed previously, $A_{m}\cap [a_{4,m},b_{4,m}]\not=\ess$.
If we let $\aaa_{4,m}=\min \{ [a_{4,m},b_{4,m}]\cap A_{m}\}$ and 
$\bbb_{4,m}=\max \{[a_{4,m},b_{4,m}]\cap A_{m}\}$ then
$$g_{m}'(a_{4,m})=g_{m}'(\aaa_{4,m})=f'_{n(m)}(\aaa_{4,m})<
g_{m}'(b_{4,m})=g_{m}'(\bbb_{4,m})=f'_{n(m)}(\bbb_{4,m}).$$
Therefore, using piecewise linearity of $f_{n(m)}'$ on 
$[a_{3,m},b_{3,m}]$ we have
$$\xi''_{n(m)}\defeq f_{n(m)}''(x)>0\text{ for all }x\in (a_{3,m},b_{3,m}).$$
Using \eqref{*B8*a} we infer
\begin{equation}\label{*B14*b}
b_{4,m}-a_{4,m}>\frac{d_{n(m)}}{2m\Xi_{n(m)}}.
\end{equation}
Observe that from $a_{4,m},b_{4,m}\in A$ it follows that if
$E_{f,m}$ denotes the convex hull of $\{ (x,y): x\in [a_{4,m},b_{4,m}],\ y\geq f(x) \}$ then with $g$ defined after \eqref{*B6*a} we have $g(x)=\min
\{ y:(x,y)\in E_{f,m} \}$ for $x\in [a_{4,m},b_{4,m}]$. 
Now we can apply Lemma \ref{*lemB4*}
 with
$a=a_{4,m}$, $b=b_{4,m}$, $\xi_{0}=\xi_{n(m)}$, $\xi_{0}''=\xi_{n(m)}''$,
$f_{0}=f_{n(m)}$,  $\ddd_{0}=\ddd_{m,n(m)}$ and $f=f$,
to infer that if $a_{4,m}\leq x_{1}<x_{2}\leq b_{4,m}$ and $(x_{1},x_{2})\cap A=\ess$
then
\begin{equation}\label{*B14*c}
x_{2}-x_{1}< \TTT_{m}\defeq 4\sqrt{\frac{\ddd_{m,n(m)}}{\xi_{n(m)}}}.
\end{equation}
Observe that by \eqref{*B3*b}, $\TTT_{m}\to 0$ as $m\to\oo$.
We can suppose that $\ddd_{m,n}$ is chosen so small that
\begin{equation}\label{*63b*}
10^{\frac{2m\Xi_{n}}{d_{n}}}<\frac{1}{4}\sqrt{\frac{\xi_{n}}{\ddd_{m,n}}},
\text{ that is }
\ddd_{m,n}<\frac{1}{16}10^{-\frac{4m\Xi_{n}}{d_{n}}}\xi_{n}.
\end{equation}
Select $\ell _{m}$  such that
\begin{equation}\label{*63bb*}
10^{-\ell _{m}-1}\leq 
\TTT_{m}
\leq 10^{-\ell _{m}}.
\end{equation}
Since $\TTT_{m}\to 0$ we also have $\ell _{m}\to\oo$ as $m\to\oo$.

By \eqref{*B14*c}, the set $A$ intersects any $[(j-1)\cdot 10^{-\ell_{m}},j\cdot 10^{-\ell _{m}}]$ grid interval which is in $[a_{4,m},b_{4,m}]$.
The number of these intervals is larger than
$((b_{4,m}-a_{4,m})/10^{-\ell _{m}})-3$.

Hence if $\ell _{m}>2$ then by \eqref{*B14*b},
 \eqref{*63b*} and \eqref{*63bb*} we obtain
 that  $A$ intersects in $[a_{4,m},b_{4,m}]$ at least
$$
(b_{4,m}-a_{4,m})\cdot 10^{\ell _{m}}-3\geq
\frac{d_{n(m)}}{2m\Xi_{n(m)}}\cdot 10^{\ell _{m}}-3\geq \frac{1}{\ell _{m}+1}
10^{\ell _{m}}-3\geq \frac{1}{\ell _{m}+1}
10^{\ell _{m}-1}$$ many intervals
of the form $[(j-1)\cdot 10^{-\ell_{m}},j\cdot 10^{-\ell _{m}}]$.
Hence
$$\limsup_{m\to\oo}\frac{\log \can_{10,\ell_{m}}(A)}{\ell _{m}\log 10}\geq \limsup_{m\to\oo}\frac{\log (\frac{1}{\ell_m+1}10^{\ell _{m}-1})}{\ell _{m}\log 10}=1.$$
This implies that $\udimm A=1.$

\end{proof}

\begin{theorem}\label{*th6}
Suppose $1<\aaa\leq 2.$ If $f\in C^{\aaa}[0,1]$ then there is a closed set $A\sse[0,1]$  such that $\dimh A=\aaa-1$ and $f|_{A}$ is convex, or concave on $A$.
\end{theorem}

 \begin{proof}
 If $\aaa=2$ then $f\in C^{2}[0,1]$ and there is an interval
 $A\sse [0,1]$  such that $f''$ is not changing its sign.
 
 Next suppose that $1<\aaa<2$.
 Denote by $E_{f}$ the convex hull of the graph of $f$, that is,
 the convex hull of $\ds \{ (x,f(x)):x\in[0,1] \}$.
 Set 
 $$\fff_{1,f}(x)=\max\{y:(x,y)\in E_{f}  \}\text{ and }
 \fff_{2,f}(x)=\min\{y:(x,y)\in E_{f}  \}.$$
 If $\fff_{1,f}$ and $\fff_{2,f}$ both coincide with the line segment connecting
 $(0,f(0))$ and $(1,f(1))$ then $f$ also coincides with them and hence it is linear and $A=[0,1]$.
 
 Suppose $\fff_{2,f}$ is not a line segment connecting  $(0,f(0))$ and $(1,f(1))$,
 (the other case with $\fff_{1,f}$ is similar). Set $A=\{ x: \fff_{2,f}(x)=f(x) \}.$
Then $A\cap (0,1)\not= \ess.$
Obviously, $\fff_{2,f}$  is convex.

One can also easily see that it is $C^{\aaa}$. Indeed, 
since $f$ is differentiable $f'(x)=\fff_{2,f}'(x)$ holds for
$x\in A$. If $(c,d)$ is an interval contiguous to $A$
then $\fff_{2,f}'(x)=\fff_{2,f}'(c)=\fff_{2,f}'(d)=f'(c)=f'(d)$
holds for any $x\in (c,d)$.

Suppose $y\in A$. 

If $x\in A$ then
$$|\fff_{2,f}'(x)-\fff_{2,f}'(y)|=|f'(x)-f'(y)|\leq \calla(f)|x-y|^{\aaa-1}.$$ 

If $x\not\in A$, but $\min A< x<\max A$ then we select $c,d\in A$  such that 
$x\in (c,d)$ and $(c,d)\cap A=\ess.$ Then $$|\fff_{2,f}'(x)-\fff_{2,f}'(y)|=|f'(c)-f'(y)|=|f'(d)-f'(y)|\leq$$
$$\calla(f)\min\{ |c-y|^{\aaa-1}, |d-y|^{\aaa-1} \}\leq \calla(f)
|x-y|^{\aaa-1}.$$
The cases $\min A> x$, or $x>\max A$ are similar and are left
to the reader.

Moreover, $\fff_{2,f}$ is linear on the intervals contiguous to $A$.
Then $\fff_{2,f}'$ is a $C^{\aaa-1}$ function which is monotone increasing and is constant on the intervals contiguous to $A$
and $$\fff_{2,f}'([0,1])=\fff_{2,f}'(A)=[\fff_{2,f}'(0),\fff_{2,f}'(1)]$$ and by Theorem \ref{*th3} we obtain that $\dimh A\geq \aaa-1$. 
 \end{proof}

 \begin{theorem}\label{*th12}
 Suppose $0\leq \aaa<1$ and $\cag_{\aaa}$ is a dense $G_{\ddd}$
 set in $C_{1}^{\aaa}[0,1]$, (when $\aaa=0$ then in this theorem and in its proof we use $C^{0}[0,1]$ instead of $C_{1}^{\aaa}[0,1]$). Then  there exists a dense $G_{\ddd}$ set $\cag_{1+\aaa}$
 in $C_{1}^{1+\aaa}[0,1]$, (in $C^{1}[0,1]$ when $\aaa=0$)  such that for any $f\in \cag_{1+\aaa}$
 we have $f'\in \cag_{\aaa}$.
 \end{theorem}

 \begin{proof}
 Suppose $\cag_{\aaa}=\cap_{m=1}^{\oo}\cag_{\aaa,m}$
 where $\cag_{\aaa,m}$ is dense and open  in $C_{1}^{\aaa}[0,1]$.
 Set $${\cag}_{1+\aaa,m}=\{ c+\int_{0}^{x}\fff(t)dt:c\in\R, \ \fff\in{\cag}_{\aaa,m} \}.$$
 Then ${\cag}_{1+\aaa,m}$ is dense and open in $C_{1}^{1+\aaa}[0,1]$. Set $\cag_{1+\aaa}=\cap_{m}{\cag}_{1+\aaa,m}$.
 \end{proof}

The next lemma is a variant of Proposition 3 of \cite{[crm]}.
 
\begin{lemma}\label{*lemAB*}
If $f\in C^{1}[0,1]$ and $f|_{A}$ is convex (or concave) then there is
a set $B$ such that $f'|_{B}$ is monotone and
\begin{equation}\label{*dimABest*}
\dimh B\geq \dimh A,\qquad \udimm B\geq \udimm A \text{ and }
\ldimm B \geq \ldimm A.
\end{equation}
\end{lemma}

\begin{proof}
By turning to its closure we can assume that $A$
is closed and without limiting generality we suppose that
$f|_{A}$ is convex.

 Denote by $\cai_A$ the shortest closed interval containing $A$.
One can extend the definition of $f|_{A}$ onto $\cai_A$ to obtain a convex function $h$
defined on $\cai_{A}$ which is continuous and affine on intervals
contiguous to  $A$
 and $f|_{A}=h|_{A}$. At two-sided accumulation points
$a$ of $A$ we have $h'(a)=f'(a)=g(a)$. At one-sided accumulation points $a$ of $A$
we have $h'_{\pm}(a)=f'(a)=g(a)$ where $h'_{\pm}(a)=h'_{+}(a)$, or $h'_{\pm}(a)=h'_{-}(a)$ for right, or left accumulation points of $A$, respectively.
Since Hausdorff dimension is not changed if we remove a countable
set one could simply denote by $B$ the accumulation points of
$A$ and observe that by convexity of $h$ on $\cai_A$ we have
$f'$ monotone increasing on $B$ and $\dimh B=\dimh A$.

The Minkowski dimension is more sensitive to alterations on
countable sets.
Suppose $a$ is an isolated point of $A$.
If $a$ is not an endpoint of $\cai_A$ then select $c\in A$ and $b\in A$ such that
$b<a<c$ and $(b,c)\cap A=\{ a \}.$ By the Mean Value Theorem
there is $a_{-}\in (b,a)$ and $a_{+}\in (a,c)$  such that $$f'(a_-)=\frac{f(a)-f(b)}{a-b}=\frac{h(a)-h(b)}{a-b}, \text{ and }$$
$$f'(a_+)=\frac{f(c)-f(a)}{c-a}=\frac{h(c)-h(a)}{c-a}.$$
If $a$ is the left-endpoint of $\cai_A$ then we define only $a_+$, if 
$a$ is the right-endpoint of $\cai_A$ then we define only $a_-$.

Denote by $B$ the set which contains all accumulation points of $A$
and the points $a_{+}$ and $a_{-}$ for isolated points of $A$.
Then the convexity of $h$ on $\cai_A$ and the above
equalities imply that $f'$ is monotone increasing on $B$.
The $1/2^{\ell }$ grid intervals taken into consideration in $\can_{2,\ell }(B)$
cover all accumulation points of $A$ and the points $a_{+}$ and $a_{-}$ 
corresponding to isolated points of $A$. 
Hence $\can_{2,\ell }(A)\leq  3\can_{2,\ell }(B)$.
This implies that the part of \eqref{*dimABest*} concerning
the Minkowski dimension holds as well. Since $B$ differs from
$A$ in a countable set for this set $B$ the statement about
the Hausdorff dimension holds as well.
\end{proof} 
 
\begin{theorem}\label{*th11b}
Suppose $0\leq \aaa<2.$ There exists a dense $G_{\ddd}$ set $\cag$ in $C_{1}^{\aaa}[0,1]$ (in $C[0,1]$ when $\aaa=0$)
 such that for any $f\in\cag$ and $A\sse [0,1]$
 if $f|_{A}$ is convex, or concave then 
\begin{equation}\label{*C1*a}
\dimh A\leq \ldimm A\leq \max \{ 0, \aaa-1 \}.
\end{equation} 
\end{theorem}

\begin{proof}
For $1\leq \aaa <2$ apply Theorems \ref{*th4}, \ref{*th12} and Lemma \ref{*lemAB*}.

Suppose $0<\aaa<1 .$ Our argument will be a variant of the proof of
Theorem \ref{*th4}.
We prove the existence of a  $G_{\ddd}$ set $\cag$ in $C_{1}^{\aaa}[0,1]$
 such that for any $f\in\cag$ and $A\sse [0,1]$
 such that  $f|_{A}$ is convex  we have 
$\ldimm A=0.$ The concave case is similar (or take $-1$ times $\cag$), and
the intersection of two dense $G_{\ddd}$ sets is again dense $G_{\ddd}$.
This implies the claim of the theorem for the case $0<\aaa<1 $.
The case $\aaa=0$, like in the case of Theorem \ref{*th3} needs a little
adjustment, since we work in $C[0,1]$ in this case.

We select a set $\{ {{f}}_{n} \}_{n=1}^{\oo}$ of $C^{\oo}$ functions which is 
dense in $C_{1}^{\aaa}[0,1]$, (in  $C[0,1]$ when $\aaa=0$) and choose
$\Xi_{n}$ satisfying \eqref{*G2*e}, \eqref{*G2*a}, \eqref{*G2b*a}
and 
\begin{equation}\label{*5bb*}
|{{f}}_{n}''|<\Xi_{n} \text{ for  any  } x\in [0,1].
\end{equation}
(Assumption \eqref{*G2b*a} is not needed when $\aaa=0$.)

We put
\begin{equation}\label{*C2*b}
M_{m,n}=2{\Xi_{n}  (\Xi_{n}+10)(m+n)}\text{ and }M_{m,n}'=\Big \lf \frac{\log(M_{m,n})}{\log 2} \Big\rf.
\end{equation}

We define again a function $\fff_{m,n}$ periodic by $2^{-M_{m,n}'}$.
We suppose that \eqref{*G2*c} holds and this time we have
\begin{equation}\label{*C2*d}
\fff_{m,n}(2^{-M_{m,n}'}-2^{-M_{m,n}})=\frac{-1}{(m+n)\Xi_{n}}.
\end{equation}
We suppose again that $\fff_{m,n}$ is linear on $[0,2^{-M_{m,n}'}-2^{-M_{m,n}}]$
and on $[2^{-M_{m,n}'}-2^{-M_{m,n}},2^{-M_{m,n}'}]$. Hence
$\text{ for }x\in(0,2^{-M_{m,n}'}-2^{-M_{m,n}})+j\cdot 2^{-M_{m,n}'},\ j\in\Z$
\begin{equation}\label{*C3*a}
\fff_{m,n}'(x)=\frac{-1}{\Xi_{n}(m+n)(2^{-M_{m,n}'}-2^{-M_{m,n}})}<\frac{-2^{M_{m,n}'}}{\Xi_{n}(m+n)}
< -(\Xi_{n}+10).
\end{equation}

Set 
\begin{equation}\label{*C3*c}
f_{m,n}(x)={{f}}_{n}(x)+\int_{0}^{x}\fff_{m,n}(t)dt.
\end{equation}

(When $\aaa=0$ then it is clear that $f_{m,n}\in C[0,1]$.)
Next for $0<\aaa<1$ we show that $f_{m,n}\in C_{1}^{\aaa}[0,1]\cap C^{1}[0,1].$
Indeed, also using \eqref{*G2b*a}
$$\frac{|f_{m,n}(x)-f_{m,n}(y)|}{|x-y|^{\aaa}}\leq \frac{|{{f}}_{n}(x)-{{f}}_{n}(y)|+|\int_{x}^{y}\fff_{m,n}(t)dt|}{|x-y|^{\aaa}}\leq$$
$$\frac{|{{f}}_{n}(x)-{{f}}_{n}(y)|+\frac{1}{(m+n)\Xi_{n}}|x-y|}{|x-y|^{\aaa}}
\leq
\frac{\calla({{f}}_{n})|x-y|^{\aaa}+(1-\calla({{f}}_{n}))|x-y|}{|x-y|^{\aaa}}\leq 1.$$

From \eqref{*5bb*} and \eqref{*C3*a} it follows that
\begin{equation}\label{*C3*b}
f_{m,n}''(x)<-10
\end{equation}
$\text{ for }x\in (0,2^{-M_{m,n}'}-2^{-M_{m,n}})+j\cdot 2^{-M_{m,n}'},\ $ $j=0,...,2^{-M_{m,n}'}-1.$

Put 
\begin{equation}\label{*C3*d}
\ddd_{m,n} = 2^{-2M_{m,n}},\ {\cal G}_{m}=\bigcup_{n=1}^{\oo}B_{\aaa}(f_{m,n},\ddd_{m,n})
\text{ and }\cag=\bigcap_{m=1}^{\oo}{\cal G}_{m}.
\end{equation}

Then one can easily see that ${\cal G}_{m}$ is dense in $C_{1}^{\aaa}[0,1]$ (in $C[0,1]$ when $\aaa=0$) and $\cag$
is $G_{\ddd}$.

Suppose $f\in\cag$ and $A\sse [0,1]$ is such that $f|_{A}$ is convex.
Then for any $m$ there exists $n(m)$ such that
\begin{equation}\label{*C4*a}
f\in B_{\aaa}(f_{m,n(m)},\ddd_{m,n(m)})
\end{equation}
Suppose that for a $j$ we have $p_{1}<p_{2}<p_{3}$, 
$p_{1},p_{2},p_{3}\in A\cap (0,2^{-M_{m,n(m)}'}-2^{-M_{m,n(m)}})+j\cdot 2^{-M_{m,n(m)}'}$,
\begin{equation}\label{*C4*aa}
p_{2}-p_{1}\geq 2^{-M_{m,n(m)}},\text{ and }p_{3}-p_{2}\geq2^{-M_{m,n(m)}}.
\end{equation}
Denote by $L(x)$ the tangent line of $f_{m,n(m)}$ at the point
$(p_{2},f_{m,n(m)}(p_{2}))$.
By \eqref{*C3*b} and \eqref{*C3*d} we have
\begin{equation}\label{*C5*a}
f_{m,n(m)}(p_{1})<L(p_{1})-10\cdot (p_{2}-p_{1})^{2}<
\end{equation}
$$L(p_{1})-10\cdot 2^{-2M_{m,n(m)}}=L(p_{1})-10\cdot \ddd_{m,n(m)}.$$
Similarly,
\begin{equation}\label{*C5*b}
f_{m,n(m)}(p_{3})<L(p_{3})-10\cdot \ddd_{m,n(m)}.
\end{equation}
Denote by ${\overline{L}}(x)$ the line parallel to $L(x)$, but  passing through 
$(p_{2},f(p_{2}))$. Since $|f(p_{2})-f_{m,n(m)}(p_{2})|<\ddd_{m,n(m)}$
we have $|{\overline{L}}(x)-L(x)|<\ddd_{m,n(m)}$ for any $x$.
Since $|f_{m,n(m)}(p_{i})-f(p_{i})|<\ddd_{m,n(m)}$ for $i=1,3$,
we infer from \eqref{*C5*a} and \eqref{*C5*b} that
\begin{equation}\label{*C5*c}
f(p_{1})<{\overline{L}}(p_{1})\text{ and }f(p_{3})<{\overline{L}}(p_{3}).
\end{equation}

Thus $f$ is concave on $\{ p_{1},p_{2},p_{3} \}$.
Hence, for any $j$ one can cover $A\cap  ([0,2^{-M_{m,n(m)}'}-2^{-M_{m,n(m)}}]+j\cdot 2^{-M_{m,n(m)}'})$,
by less than $6$ intervals of the form
$[(j'-1)\cdot 2^{-M_{m,n(m)}},j'\cdot 2^{-M_{m,n(m)}}]$.
This implies that $A\cap [0,2^{-M_{m,n(m)}'}]+j\cdot 2^{-M_{m,n(m)}'}$
can be covered by less than $7$ intervals of the form
$[(j'-1)\cdot 2^{-M_{m,n(m)}},j'\cdot 2^{-M_{m,n(m)}}]$.
Therefore, by \eqref{*C2*b}
\begin{equation}\label{*C6*a}
\can_{2,M_{m,n(m)}}(A)< 7\cdot 2^{M_{m,n(m)}'}\leq 7\cdot M_{m,n(m)}.
\end{equation}
Which implies that
$$\liminf_{m\to\oo}\frac{\log \can_{2,M_{m,n(m)}}(A)}{M_{m,n(m)} \log 2}=0$$
and hence $\ldimm A=0$.
\end{proof}

 Theorem \ref{*th5a*} and Lemma \ref{*lemAB*} imply that for $\frac{3}{2}\leq \aaa <1$ if $f(t)=\int_{0}^{t}B(x)dx$, where $B$ is a fractional Brownian motion of Hurst index
$\aaa$ then almost surely for any $A$ with $\udimm A>\aaa-1$ the restriction
$f|_{A}$ is not convex and $f\in C^{\aaa-}[0,1].$
 
 The case $1<\aaa<\frac{3}{2}$ is more interesting.
 The following theorem is true:
 \begin{theorem}\label{*5b*}
 Let $1\leq \aaa<2$. There exists $f\in C^{\aaa}_{1}[0,1]$
  such that for any $A\sse [0,1]$ with $\udimm A>\aaa-1$,
  $f|_{A}$ neither convex, nor concave.
 \end{theorem}

 The proof of this theorem is quite technical and for the cases $1<\aaa<2$ is the subject of another paper
\cite{[crm]}.


\end{document}